\documentclass{gNST2e}

\theoremstyle{plain}
\newtheorem{thm}{Theorem}[section]
\newtheorem{lemma}[thm]{Lemma}
\newtheorem{corollary}[thm]{Corollary}
\newtheorem{proposition}[thm]{Proposition}

\theoremstyle{remark}
\newtheorem{remark}{Remark}

\usepackage{enumitem}
\usepackage{graphicx}

\newcommand{\ds} {\displaystyle}
\newcommand*\dif{\mathop{}\!\mathrm{d}}
\def\bx{{\bf{x}}}
\def\bR{\mathbb{R}} % real line

\usepackage{color}
\newcount\Comments  % 0 suppresses notes to selves in text
\newcommand{\kibitz}[2]{\ifnum\Comments=1\textcolor{#1}{#2}\fi}
\begin{document}

%\jvol{00} \jnum{00} \jyear{2015} \jmonth{March}

%\articletype{Research Article}

\title{Theoretical Grounding for
Estimation in Conditional Independence Multivariate Finite Mixture Models}

\author{Xiaotian Zhu$^{\rm a}$$^{\ast}$\thanks{$^\ast$ Email: xiaotian.zhu@abbvie.com
\vspace{6pt}} and David R. Hunter$^{\rm b}$$^{\dagger}$\thanks{$^\dagger$ Email: dhunter@stat.psu.edu
\vspace{6pt}} \\\vspace{6pt}  $^{a}${\em{Data Science and Statistics, AbbVie, North Chicago, IL}};
$^{b}${\em{Statistics Department, Pennsylvania State University, University Park,
PA}}\\}

\maketitle

\begin{abstract}
For the nonparametric estimation of multivariate finite mixture models with the conditional 
independence assumption,  we propose a new formulation of the objective function in terms of 
penalized smoothed Kullback-Leibler distance. The nonlinearly smoothed majorization-minimization (NSMM)
algorithm is derived from this perspective. An elegant representation of the NSMM algorithm is obtained
using a novel projection-multiplication operator, a more precise monotonicity property of the 
algorithm is discovered, and the existence of a solution to the main optimization problem is proved for 
the first time.


\begin{classcode}62G05; 62H30 \end{classcode}
\end{abstract}

\section{Introduction}
In recent years, several studies 
have advanced the development of estimation algorithms,
based on expectation-maximization (EM) and its generalization called majorization-minimization (MM),   
for nonparametric estimation for conditional independence multivariate finite mixture models.
The idea for these algorithms had its genesis in the stochastic EM algorithm of
\citet{bordes2007stochastic} and was later extended to a deterministic algorithm by
\citet{benaglia2009like} and \citet{benaglia2011bandwidth}.  These algorithms were placed on
a more stable theoretical foundation due to the ascent property established by
\citet{levine2011maximum}.  A detailed account of these algorithms, along with the 
related theory of parameter identifiability, is presented in the survey article by
\citet{chauveau2015semi}.
This paper follows up on this line of research, extending the theoretical foundations
of this method and deriving novel results while also simplifying their formulation.  

%Our framework produces novel results
%such as sharpening the descent property and establishing
%the existence of a solution to the main optimization 
%problem for the first time.  Also, we provide theoretical justification for the nonlinearly smoothed 
%majorization-minimization (NSMM) algorithm by deriving it from this new perspective.

%\subsection{Theoretical Background}
Conditional independence multivariate finite mixture models have fundamental importance in both statistical theory and applications; for example, 
as \citet{chauveau2015semi} point out, these models are related to the
random-effects models of \citet{laird1982random}.  The basic setup assumes that $r$-dimensional vectors 
${\bf{X}}_i = ({{X}}_{i,1}, {{X}}_{i,2}, ..., {{X}}_{i,r})^\top$, $1\leq i\leq n$, are simple random samples from a finite mixture of 
$m>1$ components with positive mixing proportions $\lambda_1, \lambda_2, ..., \lambda_m$ that sum to 1, 
and density functions $f_1, f_2, ..., f_m$. Here, we assume
$m$ is known.  For recent work that addresses the estimation of $m$, along with a different approach
to the estimation of the model parameters than the one outlined here, see
\citet{bonhomme2014nonparametric} and \citet{kasahara2014non}.

The conditional independence assumption, which arises naturally in analysis of data with repeated measurements, says each $f_j$, $1\leq j\leq m$, 
is equivalent to the product of its marginal densities $f_{j,1}, f_{j,2}, ..., f_{j,r}$. Thus, the mixture density is
\begin{equation} \label{MixtureDensity}
g({\bf x})=\sum\limits_{j=1}^{m}\lambda_jf_{j}({\bf x})
=\sum\limits_{j=1}^{m}\lambda_j\prod\limits_{k=1}^{r}f_{j,k}(x_k)
\end{equation}
for any ${\bf x}=(x_1,... x_r)^\top\in \bR^r$. 
This is often regarded as a semi-parametric model with $\lambda_1, \ldots, \lambda_m$ being the Euclidean parameters and 
$f_{j,k}$, $1\leq j\leq m$, $1 \leq k\leq r$ being the functional parameters. Let $\theta$ denote all of these parameters. 

%This model is different from the type of mixture models studied in \citet{lindsay1995mixture} where 
%the mixing distribution is completely unspecified but the component densities are known to come 
%from a parametric family.

The identifiability of the parameters in the model (\ref{MixtureDensity}) was not clear until the breakthrough in \citet{hall2003nonparametric} which established 
the identifiability when $m=2$ and $r\geq 3$. Some follow-up work appeared, for example, \citet{hall2005nonparametric} and 
\citet{kasahara2009nonparametric}, until the fundamental result that established generic identifiability of (\ref{MixtureDensity}) for $r\geq 3$ was 
obtained \citep{allman2009identifiability} based on an algebraic result of \citet{kruskal1976more, kruskal1977three}.

\citet{bordes2007stochastic} proposed a stochastic nonparametric EM algorithm (npEM) estimation algorithm for the estimation of semiparametric 
mixture models. \citet{benaglia2009like} and \citet{benaglia2011bandwidth} proposed a deterministic version of the algorithm for the estimation 
of (\ref{MixtureDensity}) and studied bandwidth slection related to it. However, all these algorithms lack an objective function as well as the descent 
property which chracterizes any traditional EM algorithm \citep{dempster1977maximum}. A significant improvement comes from \citet{levine2011maximum}, 
which proposes a smoothed likelihood as the objective function and leads to a smoothed version of the npEM that does possess the desired 
descent property. The authors point out the similarities between their approach and the one in \citet{eggermont1999nonlinear} for non-mixtures.  
However, the constraints imposed by the condition that each $f_{jk}$ must integrate to one lead to 
tricky optimization issues and necessitate a slightly awkward normalization step to satisfy these constraints.  In 
reformulating the parameter space, the current
paper removes the constraints and provides a rigorous justification for the algorithm, proving the existence
of a solution to the main optimization problem for the first time.  In addition, this paper sharpens the descent property
by deriving a positive lower bound on the size of the decrease in the objective function at each iteration.

%There is also a large literature on variations or special cases of the model (\ref{MixtureDensity}). A version where all the marginal densities are the same 
%for the same component is considered in \citet{hettmansperger2000almost}, \citet{elmore2003identifiability} and \citet{cruz2004nonparametric}. For 
%univariate cases, \citet{bordes2006semiparametric} and \citet{hunter2007inference} have considered identifiability and estimation for a univariate 
%location-shift semiparametric mixture model with symmetic component densities. \citet{silverman1986density} has studied the univariate mixture of the two 
%distributions where only one of them is unknown and its application to local false discovery rate (FDR) estimation. Further development along this line of 
%research on semiparametric mixture models which rely on special assumptions can be seen from \citet{bordes2010semiparametric}, 
%\citet{chauveau2014large}, and \citet{hohmann2013semiparametric}.

%\subsection{Motivating Application}
%
%There are motivating examples in other papers \citep{hettmansperger2000almost, elmore2004estimating, benaglia2009like, levine2011maximum} 
%but the current paper does not change the algorithm in practical terms when comparing to work by \citet{levine2011maximum} so we do not consider 
%new examples.

\section{Reframing the Estimation Problem}\label{sec:estimationtheory}

%The distinction between the estimation approach we propose here and that of \citet{levine2011maximum} may appear subtle, but actually it has 
%critical implications for subsequent studies. On one hand, the new approach eliminates the Euclidean parameter $\bm{\lambda}$ becuase it is not necessary 
%to explicitly require that the $e_j$, $1\leq j\leq m$, sum up to a proper probability density function. On the other hand, the main optimization problem 
%formulated here is truly a penalized smoothed Kullback-Leibler divergence. \citet{levine2011maximum} discussed the optimization as a penalized 
%smoothed Kullback-Leibler divergence minimization problem, but the formulation in our approach seems purer. Comparing the regularity conditions of the 
%two approaches, we have added boundedness of the kernel function of the smoother $\mathcal{N}_h$, which is used in the proof of existence of a solution 
%to the main optimization problem.

In the following, we first consider an ideal setting where the target density is known (i.e., the sample size is infinity). Then we replace the target density 
by its empirical version and obtain the discrete algorithm.

\subsection{Setup and Notation}\label{subsec:setupandnotation}
Let ${\bf{x}}=(x_1,x_2,\cdots,x_r)^\top\in \bR^r$ and
let $g$ denote a target density on $\bR^r$, with support in the interior of $\Omega$, where
$\Omega$ is a compact and convex set in $\bR^r$.
Without loss of generality, assume $\Omega$ is the closed $r$-dimensional cube $[a,b]^r$. 
We are interested in the case when $g$ is a finite mixture of products 
of fully unspecified univariate measures, with unknown mixing parameters.

We make the following assumptions:
\begin{enumerate}%[label=(\roman{*})]

\item[(i)]
Let the number of mixing components in $g$ be fixed and denoted by $m$. 
%And $g(x)$ takes the form described in the following for its candidate, $g^*(x)$. In particular, 
There exist non-negative functions $e_j(\bx)$, $1\leq j\leq m$, such that
\begin{equation}\label{twopointone}
g(\bx)=\sum\limits_{j=1}^{m}e_j(\bx).
\end{equation}

\item[(ii)]
For each $1\leq j\leq m$, 
\begin{equation}\label{eq:CondInd}
e_j({\bx})=\theta_j\prod\limits_{k=1}^{r}e_{j,k}(x_k),
\end{equation}
where $\theta_j>0$ and for each $k$, $1\leq k\leq r$, $e_{j,k}\in L^1(\bR)$ is positive with support in $[a,b]$. Hence each $e_j({\bf x})$ is in $L^1(\bR^r)$, 
positive, and with support in $\Omega$.

\end{enumerate}

Given a bandwidth $h\in \bR$, let $s_h(\cdot,\cdot)\in L^1(\bR\times \bR)$ be nonnegative and with support in $[a,b]\times[a,b]$, such that
\begin{enumerate}

\item[(iii)]
For $v,z \in \bR$,
\begin{equation}
\int s_h(v,u) \dif u=\int s_h(u,z) \dif u=1.
\end{equation}

\item[(iv)]
There exist positive numbers $M_1(h)$ and $M_2$ such that for any $v,z \in [a,b]$,
\begin{equation}
M_1(h)\leq s_h(v,z) \leq M_2.
\end{equation}

\item[(v)]
The function $s_h$ has continuous first-order partial derivatives on $(a,b)\times(a,b)$ and there exists a constant $B$ such that for any $u,x \in (a,b)$,
\begin{equation}
\left|{\frac{\partial}{\partial v}s_h(v,z)|_{v=u}}\right| \leq B \qquad\quad \text{and}\qquad \left|{\frac{\partial}{\partial z}s_h(v,z)|_{z=x}}\right| \leq B.
\end{equation}

\item[(vi)]
If we define $f_j({\bf x})=e_j({\bf x})/\int{e_j({\bf{z}})d{\bf{z}}}$, then
\begin{equation}
f_j({\bf x})\geq (M_1(h))^r,
\end{equation}
for all ${\bf x}\in\Omega$ and for each $j\in\{1,2,...,m\}$.

%\item[(vii)]
%Let $\lambda_j=\int{e_j({\bf{x}})d{\bf{x}}}>0$ for each $j$, $1\leq j\leq m$. Let ${\bf e}$ denote the $m$-tuple $(e_1,e_2,\cdots,e_m)$ 
%and define $f_j=e_j/\lambda_j$ so that
%\begin{equation}
%g(x)=\sum\limits_{j=1}^{m}e_j(x)=\sum\limits_{j=1}^{m}{\lambda_j f_j(x)},
%\end{equation}
%where we assume that each $f_j$ is bounded below by $(M_1(h))^r$. 
%where $M_1(h)$ is the lower bound of $s_h$ which will be described in assumption (v). 
%Later we will show that this property is preserved by the one-step updating operator $G$ of the algorithm developed in Chapter~\ref{chpt:componentica}.

\end{enumerate}

Before stating the optimization problem, we define the smoothing operators $S_h$, $S_h^*$,
and $\mathcal{N}_h$, as follows.

For any $f\in L^1(\bR^r)$, let
\begin{equation}
(S_hf)({\bf{x}})=\int\tilde{s}_h({\bf{x}},{\bf{u}})f({\bf{u}})\dif {\bf{u}} 
\quad\text{and}\quad (S_h^*f)({\bf{x}})=\int\tilde{s}_h({\bf{u}},{\bf{x}})f({\bf{u}})\dif {\bf{u}},
\end{equation}
where
\begin{equation}
\tilde{s}_h({\bf{x}},{\bf{u}})=\prod\limits_{k=1}^{r}s_h(x_k,u_k) \qquad \text{for } {\bf{x}}, {\bf{u}} \in \bR^r.
\end{equation}
Furthermore, let
\begin{equation}\label{defmathcalN}
(\mathcal{N}_hf)({\bf{x}})=\begin{cases} 
\exp[(S_h^*\log{f})({\bf{x}})]  & \text{for }{\bf{x}}\in\Omega,\\
0  & \text{ elsewhere}.
\end{cases}
\end{equation}
These smoothing operators are well-known and have many desirable
 properties \citep{eggermont1999nonlinear}. 
For instance, Lemma 1.1 of \citet{eggermont1999nonlinear} states that
for any nonnegative functions $g_1$ and $g_2$ in $L^1(\bR^r)$,
\begin{equation}
KL\left(S_hg_1, S_hg_2\right)\leq KL\left(g_1,g_2\right),
\end{equation}
where $KL$ is the Kullback-Leibler divergence defined by
\begin{equation}\label{KLdefn}
KL\left(g_1, g_2\right)=\int{\left[g_1\log{\frac{g_1}{g_2}}+g_2-g_1\right]}.
\end{equation}

%\subsection{The Projection-Multiplication Operator $P$}
%Let us introduce notation to facilitate further discussion. For any nonnegative function $f$ on $\bR^r$ such that $\int{f}>0$, and 
%$x=(x_1,x_2,%\cdots,x_r)^\top\in \bR^r$, let the operator $P$, which factorizes $f$ as a product of marginal functions on $\bR^r$, be defined as
%\begin{equation}\label{definitionofP}
%(Pf)({\bf{x}})=\ds\frac{\left[\prod\limits_{k=1}^{r}\int\limits_{R^{r-1}}f({\bf{x}})\dif x_1\dif x_2\cdots\dif x_{k-1}\dif x_{k+1}\cdots\dif x_r\right]}{\left[\int{f}\right]^{(r-1)}}.
%\end{equation}
%When $f$ is a density on $\bR^r$, the right side of (\ref{definitionofP}) simplifies because the denominator is 1.
%\begin{equation}\label{definitionofPdensity}
%(Pf)({\bf{x}})=\prod\limits_{k=1}^{r}\int\limits_{R^{r-1}}f({\bf{x}})\dif x_1\dif x_2\cdots\dif x_{k-1}\dif x_{k+1}\cdots\dif x_r
%\end{equation}

\subsection{Main Optimization Problem}\label{sec:mainoppr}

Now, we assume conditions (i) through (vi) and propose to estimate ${\bf e}$ by minimizing the function
\begin{equation}\label{eq:OptimizationProblem}
l({\bf e})=\int g({\bf{x}})\log\left[{g({\bf{x}})}/{\sum\limits_{j=1}^{m}(\mathcal{N}_he_j)
({\bf{x}})}\right]\dif {\bf{x}}+\int\left[\sum\limits_{j=1}^{m}e_j({\bf{x}})\right]\dif {\bf{x}}
\end{equation}
subject to these conditions. In fact the only assumptions that impose any constraints on ${\bf e}$ are (ii) and (vi). 
Minimization of $l({\bf e})$ can be written equivalently as minimization of the penalized smoothed Kullback-Leibler divergence
\begin{equation}\label{eq:PenalizedKL}
KL\left(g, \sum\limits_{j=1}^{m}(\mathcal{N}_he_j)\right)+\int\left[\sum\limits_{j=1}^{m}e_j-\sum\limits_{j=1}^{m}
(\mathcal{N}_he_j)\right]({\bf{x}})\dif {\bf{x}},
\end{equation}
where in (\ref{eq:PenalizedKL}) the second term acts like a roughness penalty.

The discrete version of the optimization problem replaces $g({\bf x})\dif {\bf x}$  by $\dif G_n({\bf x})$, where $G_n$ 
is the empirical distribution function of a random sample of size $n$, and in this case we minimize
\begin{equation}\label{DiscreteObjective}
l_{\text{discrete}}({\bf e})=-\frac{1}{n}\sum\limits_{i=1}^{n}\log{\sum\limits_{j=1}^{m}(\mathcal{N}_he_j)({\bf x}_i)}
+\int\sum\limits_{j=1}^{m}e_j({\bf x}).
\end{equation}

Although we do not constrain ${\bf e}$ to require that the sum of all $e_i$ is a density as required by 
Equation~(\ref{twopointone}), this property is guaranteed by the main optimization:

\begin{thm}\label{thm:sumuptoone}
Any solution $\tilde{\bf e}$ to (\ref{eq:OptimizationProblem}) or (\ref{DiscreteObjective}) satisfies
\begin{equation}\label{eq:sumisone}
\int\sum\limits_{j=1}^{m}\tilde e_j({\bf x})=1.
\end{equation}
\end{thm}

\begin{proof}
For any fixed ${\bf e}$, differentiation shows that the function $l(\alpha {\bf e})$ is minimized at the unique value
\begin{equation}
\hat\alpha=1\left/\int\sum\limits_{j=1}^{m}e_j({\bf x})\right..
\end{equation}
%So if $\int\sum\limits_{j=1}^{m}e_j({\bf x})\neq1$, then $\hat\alpha{\bf e}^S$ can acheive a smaller value of the objective function than ${\bf e}^S$ does, 
%which contradicts the fact that ${\bf e}^S$ is a solution to (\ref{eq:OptimizationProblem}). So we must have (\ref{eq:sumisone}).
Thus if ${\bf e}$ is a minimizer then Equation~(\ref{eq:sumisone}) must hold.
%See Appendix (\ref{ProofSumuptoone}).
\end{proof}

%Note that each $\int{e^S_j}$ (integral of each mixing component of the minimizer) can be interpreted as the corresponding mixing weight.
From (\ref{eq:sumisone}), we see that for each $1\leq j\leq m$, $\int{\tilde e_j}$ can be interpreted as the mixing
weight corresponding to the $j$th mixture component.

\section{The NSMM Algorithm}\label{NSMMAlgo}

In this section, we derive an iterative algorithm, using majorization-minimization \citep{hunter2004tutorial}, to 
minimize Equation~(\ref{eq:OptimizationProblem}). The algorithm, which we refer to as 
the nonlinearly smoothed majorization-minorization (NSMM) algorithm,
coincides with that of \citet{levine2011maximum}, despite the different derivation.

\subsection{An MM Algorithm}\label{sec:mm}
%The objective functions and the parameter spaces are slightly different. And a critical optimization step of \citet{levine2011maximum} 
%is largely based on intuition, where that of the derivation in the following is rigorous and based on a special optimization scheme we discovered.

Given the current estimate ${\bf e}^{(0)}$ satisfying assumptions (ii) and (vi), let us define
\begin{equation}
w^{(0)}_j({\bf{x}})=\frac{(\mathcal{N}_he^{(0)}_j)({\bf{x}})}{\sum\limits_{j'=1}^{m}(\mathcal{N}_he^{(0)}_{j'})({\bf{x}})}
\end{equation}
for $1\le j\le m$, noting that $\sum_j w^{(0)}_j({\bf x}) = 1$.
The concavity of the logarithm function gives
\begin{align}
l({\bf e})&-l({\bf e}^{(0)})\nonumber\\
%&=-\int{g({\bf{x}})\log{\frac{\sum\limits_{j=1}^{m}(\mathcal{N}_he_{j})({\bf{x}})}{\sum\limits_{j'=1}^{m}(\mathcal{N}_he^{(0)}_{j'})({\bf{x}})}}}\dif {\bf{x}}
%+\int{\left(\sum\limits_{j=1}^{m}e_j-\sum\limits_{j=1}^{m}e^{(0)}_j\right)}\nonumber\\
&=-\int{g({\bf{x}})\log{\sum\limits_{j=1}^{m}\frac{(\mathcal{N}_he^{(0)}_j)({\bf{x}})}{\sum\limits_{j'=1}^{m}
(\mathcal{N}_he^{(0)}_{j'}){\bf{x}}}\cdot\frac{(\mathcal{N}_he_j)({\bf{x}})}{(\mathcal{N}_he^{(0)}_j)({\bf{x}})}}}\dif 
{\bf{x}}+\int{\left(\sum\limits_{j=1}^{m}e_j-\sum\limits_{j=1}^{m}e^{(0)}_j\right)}\nonumber\\
&\leq-\int{g({\bf{x}})\sum\limits_{j=1}^{m}\frac{(\mathcal{N}_he^{(0)}_j)({\bf{x}})}{\sum\limits_{j'=1}^{m}(\mathcal{N}_he^{(0)}_{j'})
({\bf{x}})}\cdot\log{\frac{(\mathcal{N}_he_j)({\bf{x}})}{(\mathcal{N}_he^{(0)}_j)({\bf{x}})}}}\dif {\bf{x}}+\int{\left(\sum\limits_{j=1}^{m}
e_j-\sum\limits_{j=1}^{m}e^{(0)}_j\right)}.
\end{align}
So if we let
\begin{equation}
b^{(0)}({\bf e})=-\int{g({\bf{x}})\sum\limits_{j=1}^{m}w^{(0)}_j({\bf{x}})\cdot\log{{(\mathcal{N}_he_j)({\bf{x}})}}}\dif {\bf{x}}+\int{\left(\sum\limits_{j=1}^{m}e_j\right)},
\end{equation}
we obtain
\begin{equation}\label{bmajorizel}
l({\bf e})-l({\bf e}^{(0)})\leq b^{(0)}({\bf e})-b^{(0)}({\bf e}^{(0)}).
\end{equation}

Using the MM algorithm terminology of \citet{hunter2004tutorial}, Inequality (\ref{bmajorizel}) means that 
$b^{(0)}$ may be said to majorize $l$ at ${\bf e}^{(0)}$, up to an additive constant. 
Minimizing $b^{(0)}$ therefore yields a function ${\bf e}^{(1)}$ satisfying
\begin{equation}\label{mmdescent}
l({\bf e}^{(1)})\leq l({\bf e}^{(0)}).
\end{equation}
Thus, we now consider how to minimize $b^{(0)}({\bf e})$, subject to the assumptions on ${\bf e}$ that were stated at the beginning. 
This is to be done component-wise. That is, for each $j$, we wish to minimize
{\allowdisplaybreaks
\begin{align}
b_j^{(0)}({\bf e}) &
%=-\int{g({\bf{x}})w^{(0)}_j({\bf{x}})\cdot\log{{(\mathcal{N}_he_j)({\bf{x}})}}}\dif {\bf{x}}+\int{e_j}\\
=-\int{g({\bf{x}})w^{(0)}_j({\bf{x}})\cdot\int{\tilde{s}_h({\bf{u}},{\bf{x}})\log{e_j({\bf{u}})}}\dif {\bf{u}}}\dif {\bf{x}}+\int{e_j}\nonumber\\
%&=-\iint{g({\bf{x}})w^{(0)}_j({\bf{x}})\cdot{\tilde{s}_h({\bf{u}},{\bf{x}})\log{e_j({\bf{u}})}}\dif {\bf{u}}}\dif {\bf{x}}+\int{e_j({\bf{u}})}\dif {\bf{u}}\label{eq:Minimization}\\
&=-\iint{g({\bf{x}})w^{(0)}_j({\bf{x}})\cdot{\tilde{s}_h({\bf{u}},{\bf{x}})\left[\sum\limits_{k=1}^{r}\log{e_{j,k}(u_k)}+\log{\theta_j}\right]}\dif {\bf{u}}}\dif {\bf{x}}\nonumber\\
&\qquad\qquad\qquad\qquad\qquad\qquad\qquad\qquad+\int{\theta_j\prod\limits_{k=1}^{r}e_{j,k}(u_k)}\dif {\bf{u}}.\label{bje0}
\end{align}
}

Up to an additive term that does not involve any $e_{j,k}$, Expression~(\ref{bje0}) is
\begin{equation}\label{eq:thatnotinvolveeik}
-\sum\limits_{k=1}^{r}\iint{g({\bf{x}})w^{(0)}_j({\bf{x}})\cdot{s_h(u_k,x_k)\log{e_{j,k}(u_k)}}\dif u_k}\text{ d}{\bf{x}}+
\int{\theta_j\prod\limits_{k=1}^{r}e_{j,k}(u_k)}\dif {\bf{u}}.
\end{equation}
For any $k$ in $1, \ldots, r$, we can view Expression~(\ref{eq:thatnotinvolveeik}) as an integral with respect to d$u_{k}$. 
Differentiating the integrand with respect to $e_{j,k}(u_k)$ and equating the result to zero, 
Fubini's Theorem gives
\begin{equation}\label{eq:Propto1}
\hat{e}_{j,k}(u_k) \propto \int{g({\bf{x}})w^{(0)}_j({\bf{x}})\cdot{s_h(u_k,x_k)\text{ d}{\bf{x}}}}.
\end{equation}

This tells us, according to (\ref{eq:CondInd}), that
\begin{equation}\label{eq:Propto2}
\hat{e}_{j}({\bf{u}}) =\alpha_j \prod\limits_{k=1}^{r}\int{g({\bf{x}})w^{(0)}_j({\bf{x}})\cdot{s_h(u_k,x_k)\text{ d}{\bf{x}}}}
\end{equation}
for some constant $\alpha_j$.
%, $\hat{e}_{j}({\bf{u}})=\alpha_j$.
%\beta_j({\bf{u}})$ and
%\begin{equation}\label{eq:Propto3}
%\beta_j({\bf{u}})=\prod\limits_{k=1}^{r}\int{g({\bf{x}})w^{(0)}_j({\bf{x}})\cdot{s_h(u_k,x_k)\text{ d}{\bf{x}}}}.
%\end{equation}
To find $\alpha_j$, we plug (\ref{eq:Propto2}) into (\ref{bje0}) and differentiate with respect to $\alpha_j$, which gives
as a final result
%\begin{equation}
%-\iint{g({\bf{x}})w^{(0)}_j({\bf{x}})\cdot{\tilde{s}_h({\bf{u}},{\bf{x}})[\log{\alpha_j}+\log{\beta_j({\bf{u}})}]}\dif {\bf{u}}}\text{ d}{\bf{x}}
%+\alpha_j\int{\beta_j({\bf{u}})}\dif {\bf{u}}.
%\end{equation}
%Differentiating with respect to $\alpha_j$ gives
%\begin{equation}
%-\frac{\iint{g({\bf{x}})w^{(0)}_j({\bf{x}})\cdot{\tilde{s}_h({\bf{u}},{\bf{x}})}\dif {\bf{u}}}\text{ d}{\bf{x}}}{\alpha_j}+\int{\beta_j({\bf{u}})}\dif {\bf{u}},
%\end{equation}
%which implies that
%\begin{equation}
%\alpha_j=\frac{1}{\left[\int{g({\bf{x}})w^{(0)}_j({\bf{x}})}\text{ d}{\bf{x}}\right]^{r-1}},
%\end{equation}
%and hence
\begin{equation}\label{eq:eihat}
\hat{e}_{j}({\bf{u}})=\frac{\prod\limits_{k=1}^{r}\int{g({\bf{x}})w^{(0)}_j({\bf{x}})
\cdot{s_h(u_k,x_k)\text{ d}{\bf{x}}}}}{\left[\int{g({\bf{x}})w^{(0)}_j({\bf{x}})}\text{ d}{\bf{x}}\right]^{r-1}}.
\end{equation}
%Equation (\ref{eq:eihat}) gives each majorization-minimization step.

To summarize, our NSMM algorithm starts with some initial estimate $e^{(0)}$  
satisfying assumptions (ii) and (vi), then iterates according to
\begin{equation}
e^{(p+1)}({\bf{u}})=G(e^{(p)})({\bf{u}}),
\end{equation}
where $G(\cdot)$ performs the one-step update of Equation (\ref{eq:eihat}). 
%Note that $G(\cdot)$ can also be expressed concisely using the $P$ and $S_h$ operators:
%\begin{equation}\label{onestepupdate}
%\left(G(e^{(p)})\right)_j({\bf{u}})=\left[P\circ S_h(g\cdot w^{(p)}_j)\right]({\bf{u}}).
%\end{equation}
%That is, for $1\leq j\leq m$ and ${\bf u}\in \Omega$,
%\begin{equation}\label{onestepupdate}
%\left(G(e^{(p)})\right)_j({\bf{u}})=\frac{\prod\limits_{k=1}^{r}\int{g({\bf{x}})w^{(p)}_j({\bf{x}})\cdot{s_h(u_k,x_k)\text{ d}{\bf{x}}}}}
%{\left[\int{g({\bf{x}})w^{(p)}_j({\bf{x}})}\text{ d}{\bf{x}}\right]^{r-1}}=\left[P\circ S_h(g\cdot w^{(p)}_j)\right]({\bf{u}}),
%\end{equation}
%where
%\begin{equation}
%w^{(p)}_j({\bf{x}})=\frac{(\mathcal{N}_he^{(p)}_j)({\bf{x}})}{\sum\limits_{j'=1}^{m}(\mathcal{N}_he^{(p)}_{j'})({\bf{x}})}
%\end{equation}
%and the $P$ operator is defined later in (\ref{definitionofP}). 
In practical terms, NSMM is identical to the non-parametric maximum smoothed likelihood algorithm proposed 
in \citet{levine2011maximum}. However, our derivation uses a simpler parameter space and the normalization involved 
in each step of the algorithm is now a result of optimization. 
We have thus rigorously derived the NSMM algorithm as a special case of the majorization-minimization method.
%thanks to the special optimization landscape and scheme for each component of the mzjorized function, while in (\ref{definitionofP}) this critical 
%optimization step is largely based on intuition. The following equations summarize the correspondence:
%{\begin{eqnarray}
%\lambda^{(p+1)}_j=\int{e^{(p+1)}_j({\bf{u}})}\dif {\bf{u}}=\int{g({\bf{x}})w^{(p)}_j({\bf{x}})}\text{ d}{\bf{x}}\label{eq:tobeextended}\\
%f^{(p+1)}_{j,k}(u_k)=\displaystyle\frac{\int g({\bf{x}})w^{(p)}_j({\bf{x}})\cdot{s_h(u_k,x_k)\text{ d}{\bf{x}}}}{\int{g({\bf{x}})
%w^{(p)}_j({\bf{x}})}\text{ d}{\bf{x}}}  \label{eq:willhavesimilar}
%\end{eqnarray}}
%According to Remark (\ref{DiscreteRemark}) and the objective function (\ref{DiscreteObjective}) 

In the discrete case, we replace the density $g(\cdot)$ by the empirical distribution defined by the sample; thus,
the algorithm iterates according to the following until convergence, assuming $e^{(p)}$ is the current step estimate:

{\bf Majorization Step}: For $1\leq i\leq n$, $1\leq j\leq m$, compute\\
\begin{equation}\label{eq:discreteupdate1}
w^{(p)}_j({\bf{x}}_i)=\frac{(\mathcal{N}_he^{(p)}_j)({\bf{x}}_i)}{\sum\limits_{j=1}^{m}(\mathcal{N}_he^{(p)}_j)({\bf{x}}_i)}.
\end{equation}

{\bf Minimization Step}:  Let
\begin{equation}\label{eq:discreteupdate2}
e^{(p+1)}_j({\bf u}) = \displaystyle\frac{\prod\limits_{k=1}^{r}\sum\limits_{i=1}^{n}\frac{1}{n}w^{(p)}_j({\bf{x}}_i)s_h(u_k,x_{ik})}
{\left(\sum\limits_{i=1}^{n}\frac{1}{n}w^{(p)}_j({\bf{x}}_i)\right)^{r-1}}.
\end{equation}

%\begin{remark}
%For the $\mathcal{N}$ operator in the Majorization Step (\ref{eq:discreteupdate1}), we will need numerical convolution.
%\end{remark}

\subsection{The Projection-Multiplication Operator}\label{sec:projection}
The NSMM algorithm of Section~\ref{sec:mm} can be summarized in an elegant way using the projection-multiplication operator,
defined as follows. For any nonnegative function $f$ on $\bR^r$ such that $\int{f}>0$, and $x=(x_1,x_2,\cdots,x_r)^\top\in \bR^r$, 
let the operator $P$, which factorizes $f$ as a product of marginal functions on $\bR^r$, be defined by
\begin{equation}\label{definitionofP}
(Pf)({\bf{x}})=\ds\frac{\left[\prod\limits_{k=1}^{r}\int\limits_{R^{r-1}}f({\bf{x}})\dif x_1\dif x_2\cdots\dif x_{k-1}\dif x_{k+1}
\cdots\dif x_r\right]}{\left[\int{f}\right]^{(r-1)}}.
\end{equation}
When $f$ is a density on $\bR^r$, the right side of (\ref{definitionofP}) simplifies because the denominator is 1.
%\begin{equation}\label{definitionofPdensity}
%(Pf)({\bf{x}})=\prod\limits_{k=1}^{r}\int\limits_{R^{r-1}}f({\bf{x}})\dif x_1\dif x_2\cdots\dif x_{k-1}\dif x_{k+1}\cdots\dif x_r
%\end{equation}
As the next lemma points out, the $P$ operator commutes with the $S_h$ Operator.

\begin{lemma}\label{Commutativity}
Assume $f$ is an integrable nonnegative function on $\bR^r$ with support in a compact set $\Omega$. We have
\begin{equation}
(P \circ S_h)f=(S_h \circ P)f.
\end{equation}
\end{lemma}

\begin{proof}
See Appendix (\ref{ProofCommutativity}).
\end{proof}

Lemma~\ref{Commutativity} implies that $G(\cdot)$, which performs the one-step update of the NSMM algorithm, 
can be expressed concisely as
\begin{equation}\label{onestepupdate}
\left(G(e^{(p)})\right)_j({\bf{u}})=\left[P\circ S_h(g\cdot w^{(p)}_j)\right]({\bf{u}})
\end{equation}
for $1\le j\le m$.   In the discrete or finite-sample case, $g(\cdot)$ places weight $1/n$ at each sampled point.
Equation~(\ref{onestepupdate}) therefore suggests a geometric intuition of $G(\cdot)$ in the discrete case, 
which is illustrated in Figure~\ref{illustrateoperatorG}.

\begin{figure}[htb]
    \centering
    \includegraphics[height=3.6in]{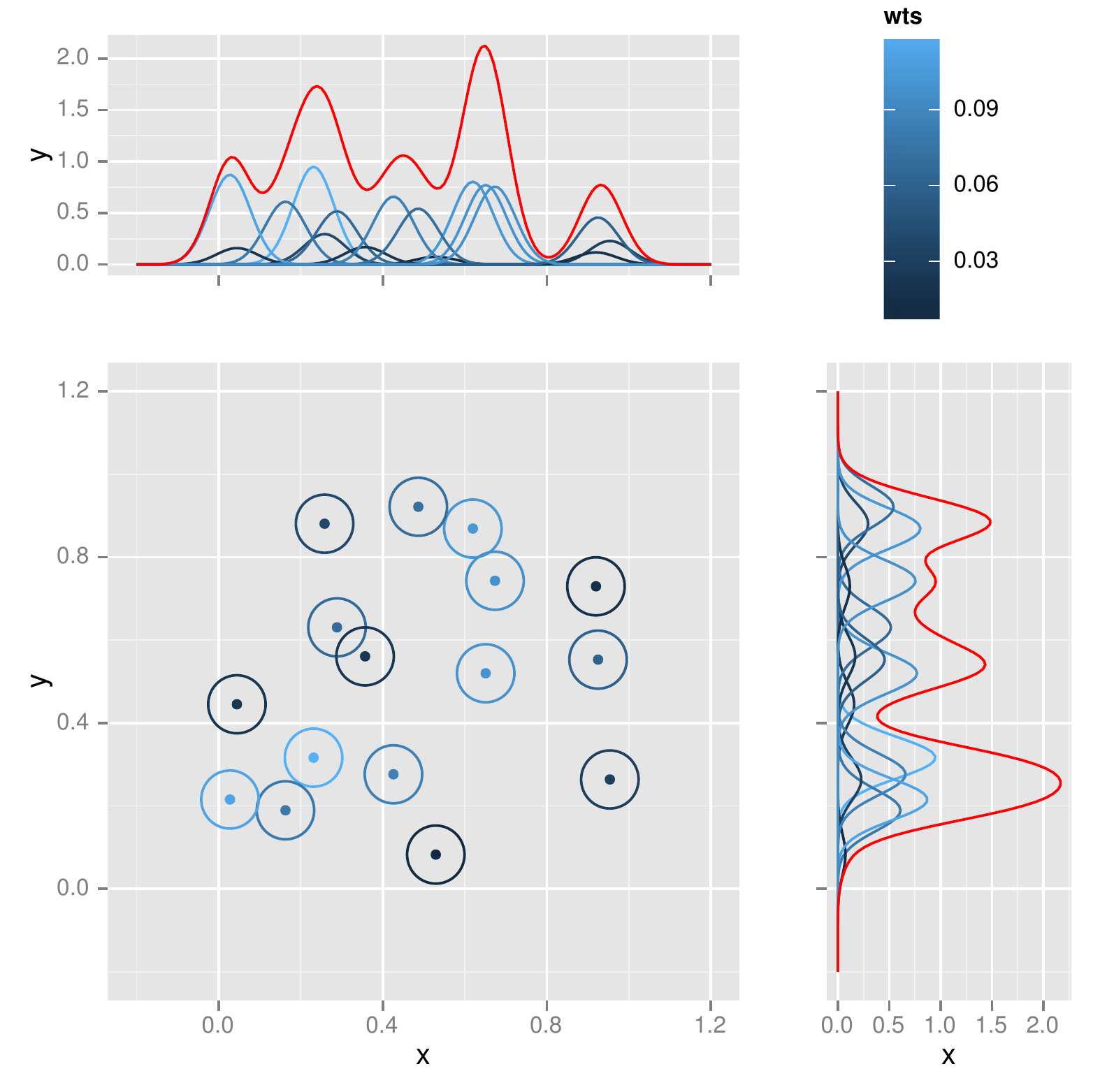}
    \caption{Illustration of the $G(\cdot)$ operator for a finite ($n=16$) sample in the case $r=2$:  
    The operator first smoothes the weighted dataset and then applies the $P$ operator to it, 
    yielding the product of the smoothed marginals, shown here in red, as the density estimator at the next iteration.}
    \label{illustrateoperatorG}
\end{figure}

\subsection{Sharpened Monotonicity}

For any MM algorithm, including any EM algorithm \citep{dempster1977maximum}, the well-known
monotonicity property of Inequality~(\ref{mmdescent}) says that the value of the objective function moves, 
at each iteration, toward the direction of being optimized  \citep{hunter2004tutorial}. 
For the NSMM algorithm, this descent property was first proved in \citet{levine2011maximum}. 
%Using our notation, this descend property says in the NSMM algorithm, at any step $p$,
%\begin{equation}
%l({\bf e}^{(p)})-l({\bf e}^{(p+1)})\geq 0.
%\end{equation}
In Proposition~\ref{MonotonicityEquation}, we present a novel result that strengthens Inequality~(\ref{mmdescent}) by giving an explicit
formula for the nonnegative value $l({\bf e}^{(p)})-l({\bf e}^{(p+1)})$.

\begin{proposition}\label{MonotonicityEquation}
In the continuous (infinite-sample) version of the NSMM algorithm, at any step $p$, we have
\begin{equation}\label{eq:monotonicityequation}
l({\bf e}^{(p)})-l({\bf e}^{(p+1)})=\sum\limits_{j=1}^{m}KL(e^{(p+1)}_j,e^{(p)}_j)+\sum\limits_{j=1}^{m}
KL(g\cdot w^{(p)}_j,g\cdot w^{(p+1)}_j).
%\sum\limits_{i=1}^{m}\int{g({\bf{x}})w^{(p)}_i({\bf{x}})\log{\frac{w^{(p)}_i({\bf{x}})}{w^{(p+1)}_i({\bf{x}})}}\dif x}
\end{equation}
\end{proposition}

\begin{proof}
See Appendix~(\ref{ProofMonotonicityEquation}).
\end{proof}
\begin{remark}\label{DiscreteMonotonicity}
The discrete version of Proposition~\ref{MonotonicityEquation} is
\begin{equation}
l({\bf e}^{(p)})-l({\bf e}^{(p+1)})=\sum\limits_{j=1}^{m}KL(e^{(p+1)}_j,e^{(p)}_j)+
\frac1n \sum\limits_{i=1}^{n}\sum\limits_{j=1}^{m} w^{(p)}_j({\bf x}_i)\log{\frac{w^{(p)}_j({\bf x}_i)}{w^{(p+1)}_j({\bf x}_i)}}.
\end{equation}
\end{remark}

Proposition~\ref{MonotonicityEquation} implies the following corollary:
\begin{corollary}
\label{MonotonicityProperty1}
In the NSMM algorithm, at any step $p$, we have
 \begin{equation}\label{eq:monotonicity1}
l({\bf e}^{(p)})-l({\bf e}^{(p+1)})\geq\sum\limits_{j=1}^{m}KL(e^{(p+1)}_j,e^{(p)}_j).
\end{equation}
\end{corollary}

%Inequality (\ref{eq:monotonicity1}) follows directly from Equation 
%(\ref{eq:monotonicityequation}) due to the nonnegativity of Kullback-Leibler divergence. However, this inequality 
Inequality (\ref{eq:monotonicity1}) may be established directly, using Jensen's Inequality, and we include this proof 
separately as an appendix because it is interesting in its own right.
\begin{proof}
Direct proof of Corollary~\ref{MonotonicityProperty1} can be found in Appendix~(\ref{ProofDirectForCorollary}).
\end{proof}

Corollary~\ref{MonotonicityProperty1} implies the following two novel results.
First, Corollary~\ref{MinIsFixed} guarantees that we only need to search among fixed point(s) of the NSMM 
algorithm for a solution to the minimization problem. This gives a theoretical basis for using the NSMM algorithm 
for this estimation problem.
\begin{corollary}\label{MinIsFixed}
Any minimizer ${\bf e}$ of $l({\bf e})$ or $l_{\text{discrete}}({\bf e})$ is a fixed point of the corresponding NSMM algorithm.
\end{corollary}
\begin{proof}
Since the right side of (\ref{eq:monotonicity1}) is strictly positive when $e^{(p+1)}_j\neq e^{(p)}_j$ for any $j$, 
a necessary condition for ${\bf e}^{(p)}$ to minimize $l({\bf e})$ is that ${\bf e}^{(p+1)}= {\bf e}^{(p)}$, i.e., that ${\bf e}^{(p)}$ is a fixed point of the algorithm.
The same reasoning works for $l_{\text{discrete}}({\bf e})$.
%See Appendix (\ref{ProofMinIsFixed}).
\end{proof}
\noindent

Second, Corollary~\ref{L1DistanceDecreases} ensures among other things that the $L^1$ distance between estimates of 
adjacent steps from an NSMM sequence will tend to zero, a result that is used in the next section.
\begin{corollary}\label{L1DistanceDecreases}
In the NSMM algorithm, at any step $p$, we have
\begin{equation}
l({\bf e}^{(p)})-l({\bf e}^{(p+1)})\geq\sum\limits_{j=1}^{m}\frac{1}{4}\left\| e^{(p+1)}_j-e^{(p)}_j \right\|_1^2,
\end{equation}
where $\|\cdot\|_1$ denotes the $L^1$ norm.
\end{corollary}
\begin{proof}
The result follows from Inequality~(3.21) in \citet{eggermont2001maximum}, which states that
\begin{equation}
KL(g_1,g_2)\geq\frac{1}{4}\|g_1,g_2\|_1^2
\end{equation}
for functions $g_1$ and $g_2$.
%See Appendix (\ref{ProofL1DistanceDecreases}).
\end{proof}
\noindent
%This is helpful in proving a convergence property of the NSMM algorithm, as will be shown in section (\ref{convergprop}).

\section{Existence of a Solution to the Maximization Problem}\label{sec:existenceofsolution}
In this section, we verify the existence of at least one solution to the main optimization problem of Section~\ref{sec:mainoppr},
a novel result as far as we are aware.
\begin{lemma}\label{lIsBoundedBelow}
\label{lisbounded}
Given ${\bf e}$ satisfying assumption (ii), we have $l({\bf e}) \geq 1$. In the discrete case, we have $l_{\text{discrete}}({\bf e}) \geq -\log{M_2}$.
\end{lemma}

\begin{proof}
See Appendix (\ref{ProoflIsBoundedBelow}).
\end{proof}

Together, Lemma \ref{MonotonicityProperty1} and Lemma \ref{lIsBoundedBelow} imply the following corollary.
\begin{corollary}
In the NSMM algorithm, $l({\bf e}^{(p)})$ will tend to a finite limit as $p$ goes to infinity. 
This result also holds in the discrete case.
\end{corollary}

We now establish some technical results that lead to the main conclusion of this section, namely, 
the existence of a minimizer of both $l({\bf e})$ and $l_{\text{discrete}}({\bf e})$.
%For a given ${\bf e}^{(0)}$, let $\lim\limits_{p\to\infty}l({\bf e}^{(p)})$ be denoted by $\eta({\bf e}^{(0)})$. In the following discussion often we will 
%simply refer to it as $\eta$.
\begin{lemma}
\label{BoundedAndEquiContinuous}
Assume conditions (i) through (vi). For each $j$, $1\leq j\leq m$, any NSMM sequence $\{e_j^{(p)}\}_{1\leq p<\infty}$ is uniformly 
bounded and equicontinuous on $\Omega$.
This result also holds in the discrete case.
\end{lemma}
\begin{proof}
See Appendix (\ref{ProofBoundedAndEquiContinuous}).
\end{proof}

More generally, Lemma~\ref{BoundedAndEquiContinuous} implies the following result:
\begin{lemma}\label{BoundandEquicontinuousAfterG}
For ${\bf e}$ satisfying assumptions (i) through (vi), in either the discrete or the continuous case, 
for $1\leq j\leq m$ and ${\bf{u}}$, ${\bf{v}}$ $\in$ $\Omega$, we have
\begin{eqnarray}
(G({\bf e}))_j({\bf{u}})\leq M_2^r,\quad\quad\quad\quad\quad\quad\quad\\
\left|(G({\bf e}))_j({\bf{u}})-(G({\bf e}))_j({\bf{v}})\right|\leq[B\cdot M_2^{r-1}]\cdot\|{\bf{u}}-{\bf{v}}\|_1.
\end{eqnarray}

\end{lemma}

The following lemma establishes a sort of lower semi-continuity of the functional $l(\cdot)$, which will be needed in proving existence of at least 
one solution to the main optimization problem.

\begin{lemma}\label{LowerSemicontinuity}
Let $\gamma^{(p)}_j\in L^1(\bR^r)$ be nonnegative and with support in $\Omega$ for each $p$ and $j$, 
where $0\leq p\leq \infty$ and $1\leq j\leq m$. Assume each $\gamma^{(p)}_j$ uniformly converges to $\gamma^{(\infty)}_j$ in 
$L^1(\bR^r)$ and that all $\gamma^{(p)}_j$ are bounded from above by a constant $Q>1$. Let ${\pmb \gamma^{(p)}}$ and 
${\pmb\gamma^{(\infty)}}$ represent $(\gamma^{(p)}_1,\cdots,\gamma^{(p)}_m)$ and $(\gamma^{(\infty)}_1,\cdots,\gamma^{(\infty)}_m)$, 
respectively. Then we have
\begin{equation}
l(\pmb\gamma^{(\infty)})\leq\liminf\limits_{p\to\infty}{l(\pmb\gamma^{(p)})}.
\end{equation}
This is also true for the discrete case.
\end{lemma}

\begin{proof}
See Appendix~(\ref{ProofLowerSemicontinuity}).
\end{proof}

\begin{thm}\label{ExistenceOfSolution}
Under assumptions (i) through (vi), there exists at least one solution to the main optimization problem (\ref{eq:OptimizationProblem}). 
This is also true in the discrete case.
\end{thm}

\begin{proof}
See Appendix~(\ref{ProofExistenceOfSolution}).
\end{proof}

To conclude this section, we discuss the rationale behind assumption (vi) and related issues such as why the $\mathcal{N}_h$ 
operator is well-defined as we applied it.

\begin{lemma}\label{AllFsAreAwayFromZero}
In an NSMM sequence $\{{\bf e}^{(p)}\}_{0\leq p\leq\infty}$, 
the $e^{(p)}_j$ are all strictly positive for all $j$. 
Moreover, if we let $\lambda^{(p)}_j=\int{e^{(p)}_j}$ and $f^{(p)}_j({\bf{u}})=e^{(p)}_j({\bf{u}})/\lambda^{(p)}_j$, then
\begin{equation}
(M_1(h))^r\leq f^{(p)}_j({\bf{u}})\leq M_2^r 
\end{equation}
for all ${\bf{u}}\in \Omega$, $p>0$, and $1\leq j\leq m$.
%And
%\begin{equation}
%(M_1(h))^r\leq \mathcal{N}_hf^{(p)}_i(u)\leq M_2^r \qquad\text{ for all } u\in \Omega, p>0 \text{ and } 1\leq i\leq m
%\end{equation}
\end{lemma}

\begin{proof}
See Appendix~(\ref{ProofAllFsAreAwayFromZero}).
\end{proof}

Lemma (\ref{AllFsAreAwayFromZero}) shows why in assumption (vi) we require the marginal densities of each mixture component to be 
bounded below by $(M_1(h))^r$ and guarantees that dividing by zero never occurs in any NSMM sequence.

\section{Discussion}

Starting from the conditional independence finite multivariate mixture model as set forth in the work of 
\citet{benaglia2009like} and \citet{levine2011maximum}, 
this manuscript proposes an equivalent but simplified parameterization.  
This reformulation leads to a novel and mathematically coherent version of the penalized Kullback-Leibler 
divergence as the main optimization criterion for the estimation of the parameters.

In this new framework, certain constraints that were previously imposed on the parameter space may be eliminated, and the solutions
obtained may be shown to follow these constraints naturally.  These contributions help to rigorously justify the non-parametric maximum smoothed 
likelihood (npMSL) estimation algorithm established by \citet{levine2011maximum}. 

As part of our investigation, we have discovered several new results, including a sharper monotonicity property of the NSMM algorithm that could ultimately
contribute to future investigations of the true 
convergence rate or other asymptotic properties of the algorithm.  We also prove, for the first time, the existence of at 
least one solution for the estimation problem of this model.

Because of the elegant simplicity and mathematical tractability 
associated with this framework, we believe the results herein
will serve as the basis for future research on this useful nonparametric model.

\appendix

\section{Mathematical Proofs}\label{app}

%\subsection{Proof of Theorem \ref{thm:sumuptoone}}
%{\allowdisplaybreaks
%\begin{proof}\label{ProofSumuptoone}
%%Let ${\bf e}^S$ be a solution to (\ref{eq:OptimizationProblem}). The same idea in the following works for a solution to (\ref{DiscreteObjective}) as well.
%%Let $\alpha$ be a scalar and substitute $\alpha{\bf e}^S$ for ${\bf e}^S$ in (\ref{eq:OptimizationProblem}). We're going to find the value of $\alpha$ 
%%that minimizes (\ref{eq:OptimizationProblem}), given the fixed ${\bf e}^S$. To this end, differentiating the objective function with respect to $\alpha$ 
%%and we find that
%For any fixed ${\bf e}$, differentiation shows that the function $l(\alpha {\bf e})$ is minimized at the unique value
%\begin{equation}
%\hat\alpha=1\left/\int\sum\limits_{j=1}^{m}e_j({\bf x})\right..
%\end{equation}
%%So if $\int\sum\limits_{j=1}^{m}e_j({\bf x})\neq1$, then $\hat\alpha{\bf e}^S$ can acheive a smaller value of the objective function than ${\bf e}^S$ does, 
%%which contradicts the fact that ${\bf e}^S$ is a solution to (\ref{eq:OptimizationProblem}). So we must have (\ref{eq:sumisone}).
%Thus if ${\bf e}$ is a minimizer then Equation~(\ref{eq:sumisone}) must hold.
%\end{proof}
%}

\subsection{Proof of Lemma \ref{Commutativity}}
{\allowdisplaybreaks
\begin{proof}\label{ProofCommutativity}
Since $(P \circ S_h)$ is linear, we only need to consider the case where $f$ is a density function. By Fubini's Theorem and Equation (\ref{definitionofP}),
\begin{align*}
&\left[(P \circ S_h)f\right]({\bf{x}})\nonumber\\
%&=\left[P(S_hf)\right]({\bf{x}})\nonumber\\
%&=\prod\limits_{k=1}^{r}\int\limits_{R^{r-1}}(S_hf)({\bf{x}})\dif x_1\dif x_2\cdots\dif x_{k-1}\dif x_{k+1}\cdots\dif x_r\nonumber\\
&=\prod\limits_{k=1}^{r}\int\limits_{R^{r-1}}\left(\int\limits_{\bR^r}\tilde{s}_h({\bf x},{\bf u})f({\bf u})\dif {\bf u}\right)\dif x_1\dif 
x_2\cdots\dif x_{k-1}\dif x_{k+1}\cdots\dif x_r\nonumber\\
&=\prod\limits_{k=1}^{r}\int\limits_{R^{r}}\left(\int\limits_{R^{r-1}}\tilde{s}_h({\bf x},{\bf u})f({\bf u})\dif x_1\dif x_2\cdots\dif x_{k-1}\dif x_{k+1}
\cdots\dif x_r\right)\dif {\bf u}\nonumber\\
%&=\prod\limits_{k=1}^{r}\int\limits_{R^{r}}{s}_h(x_k,u_k)f({\bf u})\dif {\bf u}\nonumber\\
&=\prod\limits_{k=1}^{r}\int\limits_{R}{s}_h(x_k,u_k)\left(\int\limits_{R^{r-1}} f({\bf u})\dif u_1\dif u_2\cdots\dif u_{k-1}\dif u_{k+1}
\cdots\dif u_r\right)\dif u_k\nonumber\\
&=\int\limits_{R^{r}}\left(\prod\limits_{k=1}^{r}{s}_h(x_k,u_k)\right)\cdot\prod\limits_{k=1}^{r}\left(\int\limits_{R^{r-1}} f({\bf u})\dif u_1
\dif u_2\cdots\dif u_{k-1}\dif u_{k+1}\cdots\dif u_r\right)\dif {\bf u}\nonumber\\
%&=\int\limits_{R^{r}}\left(\tilde{s}_h({\bf x},{\bf u})\right)\cdot\prod\limits_{k=1}^{r}\left(\int\limits_{R^{r-1}} f({\bf u})\dif u_1\dif u_2\cdots\dif u_{k-1}\dif 
%u_{k+1}\cdots\dif u_r\right)\dif {\bf u}\nonumber\\
%&=\left[S_h(Pf)\right]({\bf{x}})\nonumber\\
&=\left[(S_h\circ P)f\right]({\bf{x}}).
\end{align*}
\end{proof}
}

\subsection{Proof of Propositon~\ref{MonotonicityEquation}}
{\allowdisplaybreaks
\begin{proof}\label{ProofMonotonicityEquation}
Direct evaluation and the definition of Kullback-Leibler divergence in Equation~(\ref{KLdefn}) give
\begin{align}
&l({\bf e}^{(p)})-l({\bf e}^{(p+1)})\nonumber\\
&=\int{g({\bf{x}})\log{\frac{\sum\limits_{c=1}^{m}(\mathcal{N}_he^{(p+1)}_c)({\bf{x}})}{\sum\limits_{d=1}^{m}(\mathcal{N}_he^{(p)}_d)({\bf{x}})}}}\dif {\bf x}
=\sum\limits_{j=1}^{m}\int{g({\bf{x}})w^{(p)}_j({\bf{x}})\log{\frac{\sum\limits_{c=1}^{m}(\mathcal{N}_he^{(p+1)}_c)({\bf{x}})}{\sum\limits_{d=1}^{m}
(\mathcal{N}_he^{(p)}_d)({\bf{x}})}}}\dif {\bf x}\nonumber\\
&=\sum\limits_{j=1}^{m}\int{g({\bf{x}})w^{(p)}_j({\bf{x}})\log{\frac{w^{(p)}_j({\bf{x}})}{w^{(p+1)}_j({\bf{x}})}\cdot\frac{(\mathcal{N}_he^{(p+1)}_j)({\bf{x}})}
{(\mathcal{N}_he^{(p)}_j)({\bf{x}})}}}\dif {\bf x}\nonumber\\
&=\sum\limits_{j=1}^{m}\int{g({\bf{x}})w^{(p)}_j({\bf{x}})\log{\frac{(\mathcal{N}_he^{(p+1)}_j)({\bf{x}})}{(\mathcal{N}_he^{(p)}_j)({\bf{x}})}}}\dif {\bf x}+
\sum\limits_{j=1}^{m}\int{g({\bf{x}})w^{(p)}_j({\bf{x}})\log{\frac{w^{(p)}_j({\bf{x}})}{w^{(p+1)}_j({\bf{x}})}}}\dif {\bf x}\nonumber\\
%&=\sum\limits_{j=1}^{m}KL(e^{(p+1)}_j,e^{(p)}_j)+\sum\limits_{j=1}^{m}\int{g({\bf{x}})w^{(p)}_j({\bf{x}})\log{\frac{w^{(p)}_j({\bf{x}})}{w^{(p+1)}_j({\bf{x}})}}}
%\dif {\bf x}\nonumber\\
&=\sum\limits_{j=1}^{m}KL(e^{(p+1)}_j,e^{(p)}_j)\nonumber\\
&{\quad\quad}+\sum\limits_{j=1}^{m}\int\left[{g({\bf{x}})w^{(p)}_j({\bf{x}})\log{\frac{g({\bf{x}})w^{(p)}_j({\bf{x}})}{g({\bf{x}})w^{(p+1)}_j({\bf{x}})}}}+
g({\bf{x}})w^{(p+1)}_j({\bf{x}})-g({\bf{x}})w^{(p)}_j({\bf{x}})\right]\dif {\bf x}\nonumber\\
&=\sum\limits_{j=1}^{m}KL(e^{(p+1)}_j,e^{(p)}_j)+\sum\limits_{j=1}^{m}KL(g\cdot w^{(p)}_j,g\cdot w^{(p+1)}_j).
\end{align}
\end{proof}
}

\subsection{Direct Proof of Corollary (\ref{MonotonicityProperty1})}
{\allowdisplaybreaks
\begin{proof}\label{ProofDirectForCorollary}
If we define $\lambda_j=\int{e_j({\bf x})}\dif {\bf x}$ and $f_j({\bf x})=e_j({\bf x})/\lambda_j$, then Jensen's inequality together with 
some simplification give
\begin{align*}
&l({\bf e}^{(p)})-l({\bf e}^{(p+1)})\nonumber\\
&=\int{g({\bf{x}})\log{\frac{\sum\limits_{j=1}^{m}(\mathcal{N}_he^{(p+1)}_j)({\bf{x}})}{\sum\limits_{j'=1}^{m}(\mathcal{N}_he^{(p)}_{j'})({\bf{x}})}}}\dif {\bf x}
=\int{g({\bf{x}})\log{\sum\limits_{j=1}^{m}\frac{(\mathcal{N}_he^{(p)}_j)({\bf{x}})}{\sum\limits_{j'=1}^{m}(\mathcal{N}_he^{(p)}_{j'})({\bf{x}})}\cdot
\frac{(\mathcal{N}_he^{(p+1)}_j)({\bf{x}})}{(\mathcal{N}_he^{(p)}_j)({\bf{x}})}}}\dif {\bf x}\nonumber\\
&\geq\int{g({\bf{x}})\sum\limits_{j=1}^{m}\frac{(\mathcal{N}_he^{(p)}_j)({\bf{x}})}{\sum\limits_{j'=1}^{m}(\mathcal{N}_he^{(p)}_{j'})({\bf{x}})}\cdot
\log{\frac{(\mathcal{N}_he^{(p+1)}_j)({\bf{x}})}{(\mathcal{N}_he^{(p)}_j)({\bf{x}})}}}\dif {\bf x}\nonumber\\
&=\sum\limits_{j=1}^{m}\int{g({\bf{x}})w^{(p)}_j({\bf{x}})\log{\frac{\lambda^{(p+1)}_j
\prod\limits_{k=1}^{r}\mathcal{N}_hf^{(p+1)}_{j,k}(x_k)}{\lambda^{(p)}_j\prod\limits_{k=1}^{r}\mathcal{N}_hf^{(p)}_{j,k}(x_k)}}}\dif {\bf x}\nonumber\\
%&=&\sum\limits_{i=1}^{m}\int{g({\bf{x}})w^{(p)}_i({\bf{x}})\left[\log{\frac{\lambda^{(p+1)}_i}{\lambda^{(p)}_i}}+
%\sum\limits_{k=1}^{r}\log{\frac{\mathcal{N}_hf^{(p+1)}_{i,k}(x_k)}{\mathcal{N}_hf^{(p)}_{i,k}(x_k)}}\right]}\dif {\bf x}\\
&=\sum\limits_{j=1}^{m}\int{g({\bf{x}})w^{(p)}_j({\bf{x}})\left[\log{\frac{\lambda^{(p+1)}_j}{\lambda^{(p)}_j}}+
\sum\limits_{k=1}^{r}\int s_h(u_k,x_k)\log{\frac{f^{(p+1)}_{j,k}(u_k)}{f^{(p)}_{j,k}(u_k)}\dif u_k}\right]}\dif {\bf x}\nonumber\\
&=\sum\limits_{j=1}^{m}\lambda^{(p+1)}_j\log{\frac{\lambda^{(p+1)}_j}{\lambda^{(p)}_j}}+
\sum\limits_{j=1}^{m}\sum\limits_{k=1}^{r}\int{\left(\int{g({\bf{x}})w^{(p)}_j({\bf{x}})s_h(u_k,x_k)\dif {\bf x}}\right)
\log{\frac{f^{(p+1)}_{j,k}(u_k)}{f^{(p)}_{j,k}(u_k)}\dif u_k}}\nonumber\\
&=\sum\limits_{j=1}^{m}\lambda^{(p+1)}_j\log{\frac{\lambda^{(p+1)}_j}{\lambda^{(p)}_j}}+
\sum\limits_{j=1}^{m}\sum\limits_{k=1}^{r}\int{\lambda^{(p+1)}_jf^{(p+1)}_{j,k}(u_k)\log{\frac{f^{(p+1)}_{j,k}(u_k)}{f^{(p)}_{j,k}(u_k)}\dif u_k}}\nonumber\\
&=\sum\limits_{j=1}^{m}\lambda^{(p+1)}_j\log{\frac{\lambda^{(p+1)}_j}{\lambda^{(p)}_j}}+
\sum\limits_{j=1}^{m}\int{\lambda^{(p+1)}_j\left(\prod\limits_{k=1}^{r}f^{(p+1)}_{j,k}(u_k)\right)
\log{\frac{\prod\limits_{k=1}^{r}f^{(p+1)}_{j,k}(u_k)}{\prod\limits_{k=1}^{r}f^{(p)}_{j,k}(u_k)}\dif {\bf u}}}\nonumber\\
&=\sum\limits_{j=1}^{m}\int{\lambda^{(p+1)}_j\left(\prod\limits_{k=1}^{r}f^{(p+1)}_{j,k}(u_k)\right)
\log{\frac{\lambda^{(p+1)}_j\prod\limits_{k=1}^{r}f^{(p+1)}_{j,k}(u_k)}{\lambda^{(p)}_j\prod\limits_{k=1}^{r}f^{(p)}_{j,k}(u_k)}\dif {\bf u}}}\nonumber\\
&=\sum\limits_{j=1}^{m}\int{e^{(p+1)}_j({\bf u})\log{\frac{e^{(p+1)}_j({\bf u})}{e^{(p)}_j({\bf u})}}}\dif {\bf u}\nonumber\\
&=\sum\limits_{j=1}^{m}\int{\left(e^{(p+1)}_j({\bf u})\log{\frac{e^{(p+1)}_j({\bf u})}{e^{(p)}_j({\bf u})}}+e^{(p)}_j({\bf u})-e^{(p+1)}_j({\bf u})\right)}\dif {\bf u}\nonumber\\
&=\sum\limits_{j=1}^{m}KL(e^{(p+1)}_j,e^{(p)}_j).
\end{align*}
\end{proof}
}

\subsection{Proof of Lemma \ref{lIsBoundedBelow}}
{\allowdisplaybreaks
\begin{proof}\label{ProoflIsBoundedBelow}
%Assumption (v) says that $s_h$ is bounded above by $M_2$.
For each $j$, $1\leq j\leq m$, and ${\bf x}\in\Omega$, Jensen's Inequality gives
\begin{align}
\left(\mathcal{N}_he_j\right)({\bf{x}})
%&=\exp\left[\left(S^*_h\log{e_j}\right)({\bf{x}})\right]\nonumber\\
&=\exp\left\{\int\tilde{s}_h({\bf u},{\bf x})\log{e_j({\bf u})}\dif {\bf u}\right\}\nonumber\\
&\leq\int\tilde{s}_h({\bf u},{\bf x})\exp{[\log{e_j({\bf u})}]}\dif {\bf u}\nonumber\\
&=\int\tilde{s}_h({\bf u},{\bf x})e_j({\bf u})\dif {\bf u}.
%&=&\int\prod\limits_{k=1}^{r}{s}_h(u_k,x_k)e_i(u)\dif u\\
%&\leq&\int M_2^re_i(u)\dif u\\
%&\leq&M_2^r
\end{align}
Integrate both sides with respect to $\bf{x}$, then use Fubini's Theorem to obtain
\begin{align}
\int{\left(\mathcal{N}_he_j\right)({\bf{x}})}\dif {\bf x}&
%\leq\int\!\!\!\int\tilde{s}_h({\bf u},{\bf x})e_j({\bf u})\dif {\bf u}\dif {\bf x}\nonumber\\
%&=\int\left[\int\tilde{s}_h({\bf u},{\bf x})\dif {\bf x}\right]e_j({\bf u})\dif {\bf u}\nonumber\\
\leq\int e_j({\bf u})\dif {\bf u}.
\end{align}
Summing over $j$, we get
\begin{equation}
\int\sum\limits_{j=1}^{m}(\mathcal{N}_he_j)\leq\int\sum\limits_{j=1}^{m}e_j.
\end{equation}
Therefore, all three terms on the right hand side of
\begin{equation}
l({\bf e})=KL\left(g, \sum\limits_{j=1}^{m}(\mathcal{N}_he_j)\right)+\int{g}+\left[\int\sum\limits_{j=1}^{m}e_j-\int\sum\limits_{j=1}^{m}(\mathcal{N}_he_j)\right]
\end{equation}
are nonnegative and the middle term is 1, which implies that $l(\cdot)$ is always bounded below by 1.

For discrete case, Jensen's Inequality gives
\begin{align}
l_{\text{discrete}}({\bf e})&=-\frac{1}{n}\sum\limits_{i=1}^{n}\log{\sum\limits_{j=1}^{m}(\mathcal{N}_he_j)({\bf x}_i)}+\int\sum\limits_{j=1}^{m}e_j({\bf x}_i)\nonumber\\
&\geq-\frac{1}{n}\sum\limits_{i=1}^{n}\log{\sum\limits_{j=1}^{m}(\mathcal{N}_he_j)({\bf x}_i)}\nonumber\\
&\geq-\frac{1}{n}\sum\limits_{i=1}^{n}\log{\sum\limits_{j=1}^{m}(\mathcal{S}_h^*e_j)({\bf x}_i)} \geq -\frac{1}{n}\sum\limits_{i=1}^{n}\log{M_2}=-\log{M_2}.
\end{align}
\end{proof}
}

\subsection{Proof of Lemma \ref{BoundedAndEquiContinuous}}
{\allowdisplaybreaks
\begin{proof}\label{ProofBoundedAndEquiContinuous}
In the continuous case, for $p\geq 1$ and ${\bf u}\in \Omega$,
\begin{align*}
e^{(p)}_j({\bf u})&=\frac{\prod\limits_{k=1}^{r}\int{g({\bf{x}})w^{(p-1)}_j({\bf{x}})
\cdot{s_h(u_k,x_k)\text{ d}{\bf{x}}}}}{\left[\int{g({\bf{x}})w^{(p-1)}_j({\bf{x}})}\text{ d}{\bf{x}}\right]^{r-1}}
%&\leq\frac{\prod\limits_{k=1}^{r}\left[\int{g({\bf{x}})w^{(p-1)}_j({\bf{x}})\cdot{M_2\text{ d}{\bf{x}}}}\right]}
%{\left[\int{g({\bf{x}})w^{(p-1)}_j({\bf{x}})}\text{ d}{\bf{x}}\right]^{r-1}}\nonumber\\
\leq M_2^r\cdot\left[\int{g({\bf{x}})w^{(p-1)}_j({\bf{x}})}\text{ d}{\bf{x}}\right]
\leq M_2^r.
\end{align*}
Thus $\{e_j^{(p)}\}_{1\leq p<\infty}$ is uniformly bounded. Also, for any ${\bf{u}}$ in the interior of $\Omega$,
\begin{align}
\left|\frac{\partial}{\partial u_l}e^{(p)}_j({\bf{u}})\right|&=\left|\frac{\left[\prod\limits_{k\neq l}^{r}\int{g({\bf{x}})w^{(p-1)}_j({\bf{x}}){s_h(u_k,x_k)\text{ d}{\bf{x}}}}\right]
\cdot\int{g({\bf{x}})w^{(p-1)}_j({\bf{x}}){\frac{\partial}{\partial u_l}s_h(u_l,x_l)
\text{ d}{\bf{x}}}}}{\left[\int{g({\bf{x}})w^{(p-1)}_j({\bf{x}})}\text{ d}{\bf{x}}\right]^{r-1}}\right|\nonumber\\
%&\leq\frac{\left[\prod\limits_{k\neq l}^{r}\int{g({\bf{x}})w^{(p-1)}_j({\bf{x}}){M_2\text{ d}{\bf{x}}}}\right]\cdot
%\int{g({\bf{x}})w^{(p-1)}_j({\bf{x}}){B\text{ d}{\bf{x}}}}}{\left[\int{g({\bf{x}})w^{(p-1)}_j({\bf{x}})}\text{ d}{\bf x}\right]^{r-1}}\nonumber\\
&\leq B\cdot M_2^{r-1}\cdot\left[\int{g({\bf{x}})w^{(p-1)}_j({\bf{x}})}\text{ d}{\bf{x}}\right]\nonumber\\
&\leq B\cdot M_2^{r-1}.
\end{align}

By the Dominated Convergence Theorem, the above differentiation under the integral is allowed because the term 
$|g({\bf{x}})w^{(p-1)}_j({\bf{x}}){\frac{\partial}{\partial u_l}s_h(u_l,x_l)}|$ is uniformly bounded by the integrable function $B\cdot g({\bf{x}})$.

Now by the Mean Value Theorem for functions of several variables, for any ${\bf u},{\bf v}\in \Omega$, there is some $d\in(0,1)$ such that
\begin{equation}
e^{(p)}_j({\bf{u}})-e^{(p)}_j({\bf{v}})=\nabla e^{(p)}_j[(1-d){\bf{v}}+d{\bf{u}}]\cdot({\bf{u}}-{\bf{v}}).
\end{equation}
So
\begin{equation}
\left|e^{(p)}_j({\bf{u}})-e^{(p)}_j({\bf{v}})\right|\leq[B\cdot M_2^{r-1}]\cdot\|{\bf{u}}-{\bf{v}}\|_1,
\end{equation}
which shows that $\{e_j^{(p)}\}_{1\leq p<\infty}$ is equicontinuous on $\Omega$ in the $L^1$ norm.

This proof can be readily adapted to the discrete case by replacing the integrals by summations.
\end{proof}
}

\subsection{Proof of Lemma \ref{LowerSemicontinuity}}
{\allowdisplaybreaks
\begin{proof}\label{ProofLowerSemicontinuity}
We first consider the continuous case. In the following, Fatou's Lemma will be applied twice to get the desired result. First, we show by 
Jensen's Inequality that all $\mathcal{N}_h\gamma^{(p)}_j$ are bounded from above by $Q$:
\begin{align}
\left(\mathcal{N}_h\gamma^{(p)}_j\right)({\bf{x}})
%&=\exp\left[\left(S^*_h\log{\gamma^{(p)}_j}\right)({\bf{x}})\right]\nonumber\\
&=\exp\left\{\int\tilde{s}_h({\bf u},{\bf x})\log{\gamma^{(p)}_j({\bf u})}\dif {\bf u}\right\}\nonumber\\
&\leq\exp\left\{\int\tilde{s}_h({\bf u},{\bf x})\log{Q}\dif {\bf u}\right\}
%&=\exp\left\{\log{Q}\right\}\nonumber\\
=Q.
\end{align}

Now, for any fixed value of ${\bf x}$, the nonnegative measurable function $\tilde{s}_h(\cdot,{\bf x})[Q-\log{\gamma^{(p)}_j(\cdot)}]$ converges to 
$\tilde{s}_h(\cdot,{\bf x})[Q-\log{\gamma^{(\infty)}_j}(\cdot)]$ pointwise in $L^1(\bR^r)$. These functions are allowed to attain the value $+\infty$. 
By Fatou's Lemma, we have
\begin{equation}
\liminf\limits_{p\to \infty}\int{\tilde{s}_h({\bf u},{\bf x})[Q-\log{\gamma^{(p)}_j({\bf u})}]}\dif {\bf u}\geq
\int{\tilde{s}_h({\bf u},{\bf x})[Q-\log{\gamma^{(\infty)}_j}({\bf u})]}\dif {\bf u}.
\end{equation}
%So
%\begin{equation}
%\limsup\limits_{p\to \infty}\int{\tilde{s}_h({\bf u},{\bf x})\log{\gamma^{(p)}_j({\bf u})}}\dif {\bf u}\leq\int{\tilde{s}_h({\bf u},{\bf x})\log{\gamma^{(\infty)}_j}({\bf u})}\dif {\bf u}.
%\end{equation}
Exponentiating, this implies
\begin{equation}
\limsup\limits_{p\to \infty}\exp\left\{\int{\tilde{s}_h({\bf u},{\bf x})\log{\gamma^{(p)}_j({\bf u})}}\dif {\bf u}\right\}\leq
\exp\left\{\int{\tilde{s}_h({\bf u},{\bf x})\log{\gamma^{(\infty)}_j}({\bf u})}\dif {\bf u}\right\}.
\end{equation}
That is,
\begin{equation}
\limsup\limits_{p\to\infty}\left(\mathcal{N}_h\gamma^{(p)}_j\right)({\bf{x}})\leq\left(\mathcal{N}_h\gamma^{(\infty)}_j\right)({\bf{x}}),
\end{equation}
which implies that
\begin{equation}\label{criticalliminf}
\liminf\limits_{p\to\infty}g({\bf{x}})\log\frac{g({\bf{x}})}{\sum\limits_{j=1}^{m}\left(\mathcal{N}_h\gamma^{(p)}_j\right)({\bf{x}})}\geq 
g({\bf{x}})\log\frac{g({\bf{x}})}{\sum\limits_{j=1}^{m}\left(\mathcal{N}_h\gamma^{(\infty)}_j\right)({\bf{x}})}.
\end{equation}

Since $a\log(a/b)+b-a$ is nonnegative for all $a,b\geq0$, we have
\begin{equation}
g({\bf{x}})\log\frac{g({\bf{x}})}{\sum\limits_{j=1}^{m}\left(\mathcal{N}_h\gamma^{(p)}_j\right)({\bf{x}})}\geq g({\bf{x}})-
\sum\limits_{j=1}^{m}\left(\mathcal{N}_h\gamma^{(p)}_j\right)({\bf{x}})\geq -m\cdot Q.
\end{equation}
Thus, we can rewrite (\ref{criticalliminf}) as 
\begin{equation}\label{goodliminf}
\liminf\limits_{p\to\infty}\left[g({\bf{x}})\log\frac{g({\bf{x}})}{\sum\limits_{j=1}^{m}\left(\mathcal{N}_h\gamma^{(p)}_j\right)({\bf{x}})}+
m\cdot Q\right]\geq g({\bf{x}})\log\frac{g({\bf{x}})}{\sum\limits_{j=1}^{m}\left(\mathcal{N}_h\gamma^{(\infty)}_j\right)({\bf{x}})}+m\cdot Q,
\end{equation}
so that both sides are nonnegative.

Now apply Fatou's Lemma again to obtain
\begin{align}
\liminf\limits_{p\to\infty}&\int\left[g({\bf{x}})\log\frac{g({\bf{x}})}{\sum\limits_{j=1}^{m}\left(\mathcal{N}_h\gamma^{(p)}_j\right)({\bf{x}})}+m\cdot Q\right]\dif {\bf x}\nonumber\\
&\geq\int\left[\liminf\limits_{p\to\infty}g({\bf{x}})\log\frac{g({\bf{x}})}{\sum\limits_{j=1}^{m}\left(\mathcal{N}_h\gamma^{(p)}_j\right)({\bf{x}})}+
m\cdot Q\right]\dif {\bf x}\nonumber\\
&\geq\int\left[g({\bf{x}})\log\frac{g({\bf{x}})}{\sum\limits_{j=1}^{m}\left(\mathcal{N}_h\gamma^{(\infty)}_j\right)({\bf{x}})}+m\cdot Q\right]\dif {\bf x}.
\end{align}
We conclude that
\begin{equation}\label{almostthere}
\liminf\limits_{p\to\infty}\int{ g({\bf{x}})\log\frac{g({\bf{x}})}{\sum\limits_{j=1}^{m}\left(\mathcal{N}_h\gamma^{(p)}_j\right)({\bf{x}})}}\dif {\bf x}\geq 
\int g({\bf{x}})\log\frac{g({\bf{x}})}{\sum\limits_{j=1}^{m}\left(\mathcal{N}_h\gamma^{(\infty)}_j\right)({\bf{x}})}\dif {\bf x}.
\end{equation}
The uniform convergence of $\gamma^{(p)}_j$ to $\gamma^{(\infty)}_j$ for each $j$,
%\begin{equation}
%\lim\limits_{p\to\infty}\int{\sum\limits_{j=1}^{m}\gamma^{(p)}_j({\bf{x}})}\dif {\bf x}=\int{\sum\limits_{j=1}^{m}\gamma^{(\infty)}_j({\bf{x}})}\dif {\bf x}.
%\end{equation}
together with (\ref{almostthere}), imply
\begin{align}
&\liminf\limits_{p\to\infty}\left[\int{ g({\bf{x}})\log\frac{g({\bf{x}})}{\sum\limits_{j=1}^{m}\left(\mathcal{N}_h\gamma^{(p)}_j\right)({\bf{x}})}}\dif {\bf x}+
\int{\sum\limits_{j=1}^{m}\gamma^{(p)}_j({\bf{x}})}\dif {\bf x}\right]\nonumber\\
&\qquad\qquad\geq \int g({\bf{x}})\log\frac{g({\bf{x}})}{\sum\limits_{j=1}^{m}\left(\mathcal{N}_h\gamma^{(\infty)}_j\right)({\bf{x}})}\dif {\bf x}+
\int{\sum\limits_{j=1}^{m}\gamma^{(\infty)}_j({\bf{x}})}\dif {\bf x}.
\end{align}
That is,
\begin{equation}
\liminf\limits_{p\to\infty}l(\gamma^{(p)})\geq l(\gamma^{(\infty)}),
\end{equation}
which establishes the desired lower semi-continuity.

The proof can be adpated to the discrete case by replacing the integrals with summations.
\end{proof}
}

\subsection{Proof of Theorem \ref{ExistenceOfSolution}}
{\allowdisplaybreaks
\begin{proof}\label{ProofExistenceOfSolution}
By Lemma \ref{lIsBoundedBelow}, $\tau:=\inf\{l({\bf e}) | {\bf e} \text{ satisfies assumptions (ii) and (vi)} \}$ is a finite constant. So there exists a sequence $\{\bm{\psi}^{(p)}\}_{0\leq p\leq\infty}$ satisfying assumptions (ii) and (vi) such that
\begin{equation}
\lim\limits_{p\to\infty}l(\bm{\psi}^{(p)})=\tau.
\end{equation}
By Lemma \ref{BoundandEquicontinuousAfterG}, for each $j$, $1\leq j\leq m$, the sequence $\{(G(\bm{\psi}^{(p)}))_j\}_{0\leq p\leq\infty}$ is bounded and equicontinuous.

By the Arzel\`a-Ascoli theorem, we know that $\{(G(\bm{\psi}^{(p)}))_j\}_{0\leq p\leq\infty}$ has a uniformly convergent subsequence. Applying this theorem $m$ times to $\{(G(\bm{\psi}^{(p)}))\}_{0\leq p\leq\infty}$ we can extract a subsequence that converges uniformly in every component. This subsequence also satisfies (ii) and (vi).

That is, there exists a sequence $\{(G(\bm{\psi}^{(p_k)}))\}_{0\leq k\leq\infty}$, such that, for each $j$, $1\leq j\leq m$, $\{(G(\bm{\psi}^{(p_k)}))_j\}_{0\leq k\leq\infty}$ converges uniformly to a limit function in $L^1(\bR^r)$. Denote this limit function by $\tilde{\psi}_j$. As usual, let $\tilde{\bm{\psi}}$ denote the $m$-tuples $(\tilde{{\psi}}_1,\cdots,\tilde{{\psi}}_m)$. If all components of $\tilde{\bm{\psi}}$ are nonzero, then $\tilde{\bm{\psi}}$ satisfies (iii). If not, we can split up some nonzero components of $\tilde{\bm{\psi}}$ so that all components become nonzero, which does not change the value of $l(\tilde{\bm{\psi}})$. In a word, we can assume that $\tilde{\bm{\psi}}$ satisfies (vi).

%By uniform convergence, we see that $\tilde{\psi}$ also satisfies assumption (ii).

Now, by Lemma \ref{LowerSemicontinuity} and the fact that $G$ does not increase the value of $l$ (see the proof of Lemma \ref{MonotonicityProperty1}), we have
\begin{equation}
\tau\leq l(\tilde{\bm{\psi}})\leq \lim\limits_{k\to\infty}l(G(\bm{\psi}^{(p_k)}))\leq\lim\limits_{k\to\infty}l(\bm{\psi}^{(p_k)})=\lim\limits_{p\to\infty}l(\bm{\psi}^{(p)})=\tau,
\end{equation}
so that $l(\tilde{\bm\psi})=\tau$. Apply the operator $G$ to $\tilde{\bm{\psi}}$. By Lemma \ref{MonotonicityProperty1} and the fact that $l(\tilde{\bm{\psi}})$ has already attained the infimum value in this setting, we have
\begin{equation}
0\geq l(\tilde{\bm{\psi}})-l(G(\tilde{\bm{\psi}}))\geq\sum\limits_{j=1}^{m}KL((G(\tilde{\bm{\psi}}))_j,\tilde{{\psi}}_j)\geq 0.
\end{equation}
So for each $j$, $1\leq j\leq m$, $G(\tilde{\bm\psi})_j=\tilde{\psi}_j$ in $L^1(\bR^r)$. Thus in particular, by (\ref{onestepupdate}), $\tilde{\bm\psi}$ also satisfies assumption (ii). We have proved the existence of a solution, $\tilde{\bm\psi}$, to the main optimization problem (\ref{eq:OptimizationProblem}).

As above,
the proof can readily be adapted to the discrete case.
\end{proof}
}

\subsection{Proof of Lemma \ref{AllFsAreAwayFromZero}}
{\allowdisplaybreaks
\begin{proof} \label{ProofAllFsAreAwayFromZero}
First, by assumption (vi), each $e^{(0)}_j$ is strictly positive on $\Omega$. So given any ${\bf x}\in\Omega$,
\begin{equation}
\left(\mathcal{N}_he^{(0)}_j\right)({\bf{x}})=\exp\left[\left(S^*_h\log{e^{(0)}_j}\right)({\bf{x}})\right]>0.
\end{equation}
Thus,
\begin{equation}
w^{(0)}_j({\bf{x}})=\frac{\left(\mathcal{N}_he^{(0)}_j\right)({\bf{x}})}{\sum\limits_{j=1}^{m}{\left(\mathcal{N}_he^{(0)}_j\right)({\bf{x}})}}>0,
\end{equation}
which implies
\begin{equation}
\int{g({\bf{x}})w^{(0)}_j({\bf{x}})}\text{ d}{\bf x}>0.
\end{equation}
Now, we use induction. Assume
\begin{equation}
\int{g({\bf{x}})w^{(p-1)}_j({\bf{x}})}\text{ d}{\bf x}>0.
\end{equation}
We have
\begin{align}
f^{(p)}_j(u)&=\frac{\prod\limits_{k=1}^{r}\int{g({\bf{x}})w^{(p-1)}_j({\bf{x}})\cdot{s_h(u_k,x_k)\text{ d}{\bf x}}}}{\left[\int{g({\bf{x}})w^{(p-1)}_j({\bf{x}})}\text{ d}{\bf x}\right]^{r}}
%&\leq\frac{\prod\limits_{k=1}^{r}\left[\int{g({\bf{x}})w^{(p-1)}_j({\bf{x}})\cdot{M_2\text{ d}{\bf x}}}\right]}{\left[\int{g({\bf{x}})w^{(p-1)}_j{\bf{x}})}\text{ d}{\bf x}\right]^{r}}\nonumber\\
\leq M_2^r.
\end{align}
Similarly,
\begin{eqnarray}
f^{(p)}_j(u)&=\ds\frac{\prod\limits_{k=1}^{r}\int{g({\bf{x}})w^{(p-1)}_j({\bf{x}})\cdot{s_h(u_k,x_k)\text{ d}{\bf x}}}}{\left[\int{g({\bf{x}})w^{(p-1)}_j({\bf{x}})}\text{ d}{\bf x}\right]^{r}}
%&\geq\frac{\prod\limits_{k=1}^{r}\left[\int{g({\bf{x}})w^{(p-1)}_j({\bf{x}})\cdot{M_1(h)\text{ d}{\bf x}}}\right]}{\left[\int{g({\bf{x}})w^{(p-1)}_j({\bf{x}})}\text{ d}{\bf x}\right]^{r}}\nonumber\\
\geq (M_1(h))^r.
\end{eqnarray}
Therefore,
\begin{align}
\left(\mathcal{N}_he^{(p)}_j\right)({\bf{x}})&=\int{g({\bf{x}})w^{(p-1)}_j({\bf{x}})}\text{ d}{\bf x}\cdot\exp\left[\left(S^*_h\log{f^{(p)}_j}\right)({\bf{x}})\right]\nonumber\\
&\geq\int{g({\bf{x}})w^{(p-1)}_j({\bf{x}})}\text{ d}{\bf x}\cdot(M_1(h))^r\nonumber\\
&>0.
\end{align}
We conclude that
\begin{equation}
w^{(p)}_j({\bf{x}})=\frac{\left(\mathcal{N}_he^{(p)}_j\right)({\bf{x}})}{\sum\limits_{j=1}^{m}{\left(\mathcal{N}_he^{(p)}_j\right)({\bf{x}})}}>0,
\end{equation}
which gives
\begin{equation}
\int{g({\bf{x}})w^{(p)}_j({\bf{x}})}\text{ d}{\bf x}>0.
\end{equation}
The next step of the induction follows in the same way, and the result is established.
\end{proof}
}

\bibliographystyle{gNST}
 \bibliography{Biblio-Database}

\begin{thebibliography}{18}
\providecommand{\natexlab}[1]{#1}

\bibitem[Allman et~al.(2009)]{allman2009identifiability}
Allman, E.S., Matias, C., and Rhodes, J.A. (2009), ``Identifiability of
  parameters in latent structure models with many observed variables,''
  {\itshape The Annals of Statistics}, 37, 3099--3132.

\bibitem[Benaglia et~al.(2009)]{benaglia2009like}
Benaglia, T., Chauveau, D., and Hunter, D.R. (2009), ``An {{EM}}-like algorithm
  for semi-and nonparametric estimation in multivariate mixtures,'' {\itshape
  Journal of Computational and Graphical Statistics}, 18, 505--526.

\bibitem[Benaglia et~al.(2011)]{benaglia2011bandwidth}
Benaglia, T., Chauveau, D., Hunter, D.R. et~al. (2011), ``Bandwidth selection
  in an {{EM}}-like algorithm for nonparametric multivariate mixtures,''
  {\itshape Nonparametric Statistics and Mixture Models: A Festschrift in Honor
  of Thomas P. Hettmansperger}, pp. 15--27.

\bibitem[Bonhomme et~al.(2014)]{bonhomme2014nonparametric}
Bonhomme, S., Jochmans, K., and Robin, J.M. (2014), ``Nonparametric estimation
  of finite mixtures,'' Centre for Microdata Methods and Practice (cemmap)
  Working Paper CWP11/14, London.

\bibitem[Bordes et~al.(2007)]{bordes2007stochastic}
Bordes, L., Chauveau, D., and Vandekerkhove, P. (2007), ``A stochastic {{EM}}
  algorithm for a semiparametric mixture model,'' {\itshape Computational
  Statistics \& Data Analysis}, 51, 5429--5443.

\bibitem[Chauveau et~al.(2015)]{chauveau2015semi}
Chauveau, D., Hunter, D.R., and Levine, M. (2015), ``Semi-Parametric Estimation
  for Conditional Independence Multivariate Finite Mixture Models,'' {\itshape
  Statistics Surveys}, 9, 1--31.

\bibitem[Dempster et~al.(1977)]{dempster1977maximum}
Dempster, A., Laird, N., and Rubin, D. (1977), ``Maximum likelihood from
  incomplete data via the {{EM}} algorithm,'' {\itshape Journal of the Royal
  Statistical Society. Series B (Methodological)}, pp. 1--38.

\bibitem[Eggermont and LaRiccia(2001)]{eggermont2001maximum}
Eggermont, P., and LaRiccia, V., {\itshape Maximum Penalized Likelihood
  Estimation: Volume I: Density Estimation}, Vol.~1, Springer (2001).

\bibitem[Eggermont(1999)]{eggermont1999nonlinear}
Eggermont, P. (1999), ``Nonlinear smoothing and the {{EM}} algorithm for
  positive integral equations of the first kind,'' {\itshape Applied
  mathematics \& optimization}, 39, 75--91.

\bibitem[Hall et~al.(2005)]{hall2005nonparametric}
Hall, P., Neeman, A., Pakyari, R., and Elmore, R. (2005), ``Nonparametric
  inference in multivariate mixtures,'' {\itshape Biometrika}, 92, 667--678.

\bibitem[Hall and Zhou(2003)]{hall2003nonparametric}
Hall, P., and Zhou, X.H. (2003), ``Nonparametric estimation of component
  distributions in a multivariate mixture,'' {\itshape Annals of Statistics},
  pp. 201--224.

\bibitem[Hunter and Lange(2004)]{hunter2004tutorial}
Hunter, D.R., and Lange, K. (2004), ``A tutorial on MM algorithms,'' {\itshape
  The American Statistician}, 58, 30--37.

\bibitem[Kasahara and Shimotsu(2009)]{kasahara2009nonparametric}
Kasahara, H., and Shimotsu, K. (2009), ``Nonparametric identification of finite
  mixture models of dynamic discrete choices,'' {\itshape Econometrica}, 77,
  135--175.

\bibitem[Kasahara and Shimotsu(2014)]{kasahara2014non}
Kasahara, H., and Shimotsu, K. (2014), ``Non-parametric identification and
  estimation of the number of components in multivariate mixtures,'' {\itshape
  Journal of the Royal Statistical Society: Series B (Statistical
  Methodology)}, 76, 97--111.

\bibitem[Kruskal(1976)]{kruskal1976more}
Kruskal, J.B. (1976), ``More factors than subjects, tests and treatments: an
  indeterminacy theorem for canonical decomposition and individual differences
  scaling,'' {\itshape Psychometrika}, 41, 281--293.

\bibitem[Kruskal(1977)]{kruskal1977three}
Kruskal, J.B. (1977), ``Three-way arrays: rank and uniqueness of trilinear
  decompositions, with application to arithmetic complexity and statistics,''
  {\itshape Linear algebra and its applications}, 18, 95--138.

\bibitem[Laird and Ware(1982)]{laird1982random}
Laird, N.M., and Ware, J.H. (1982), ``Random-effects models for longitudinal
  data,'' {\itshape Biometrics}, pp. 963--974.

\bibitem[Levine et~al.(2011)]{levine2011maximum}
Levine, M., Hunter, D., and Chauveau, D. (2011), ``Maximum smoothed likelihood
  for multivariate mixtures,'' {\itshape Biometrika}, 98, 403--416.

\end{thebibliography}

\end{document}